\documentclass[12pt]{amsart}

\usepackage{amsmath,amssymb,amsthm,amscd}
\usepackage[shortlabels]{enumitem}
\usepackage{mathtools}
\usepackage{color}
\usepackage{bbm}
\usepackage{tikz-cd}

\usepackage[normalem]{ulem}
\usepackage{tcolorbox}
\usepackage{mathabx}

\usepackage{hyperref}
\usepackage{float}
\hypersetup{
colorlinks=true,
linkcolor={blue},
citecolor={red!50!black}
}

\usepackage{comment}

\newtheorem{theorem}{Theorem}[section]
\newtheorem*{theorem*}{Theorem}
\newtheorem{lemma}[theorem]{Lemma}
\newtheorem*{lemma*}{Lemma}

\numberwithin{equation}{section}

\def\1z{\mathbb{1}}
\newcommand{\ev}{ev}



\begin{document}

\title[Arveson's hyprrigidity conjecture is false]{Arveson's hyperrigidity conjecture \\ is false}

	        \author{Boris Bilich}
        \address{Department of Mathematics, University of Haifa, Haifa, Israel, and Department of Mathematics, University of G\"ottingen, G\"ottingen, Germany.}
        \email{bilichboris1999@gmail.com}
        
        \author{Adam Dor-On}
        \address{Department of Mathematics, University of Haifa, Haifa, Israel.}
        \email{adoron.math@gmail.com}
				
				\thanks{A. Dor-On was partially supported by a Horizon Marie-Curie SE project no. 101086394 and a DFG Middle-Eastern collaboration project no. 529300231.}

\begin{abstract}
Arveson's hyperrigidity conjecture predicts that if the non-commutative Choquet boundary of a separable operator system $\mathcal{S}$ is the entire spectrum of its generated C*-algebra $\mathcal{B}$ then $\mathcal{S}$ is hyperrigid in $\mathcal{B}$. We provide a counterexample to the conjecture with a C*-algebra $\mathcal{B}$ of type I generated by a single operator.
\end{abstract}

\maketitle

\section{Introduction}

Arveson's hyperrigidity conjecture is the last major open conjecture in non-commutative boundary theory left by Arveson, following his foundational work on the subject \cite{Arv69, Arv72, Arv98, Arv08}. It lies at the epicenter of several advancements in non-commutative analysis \cite{DK15, Clo18a, DK+, KS22, DP22}, and is intimately related to the Arveson-Douglas essential normality conjecture \cite{Arv02, Dou06, KS15, CH18, CT21}. Many authors have attempted to prove the hyperrigidity conjecture, for instance when the generated C*-algebra is of type I \cite{Kle14}, or when it is commutative \cite{DK21}. Arveson himself provided a local version of the conjecture when the generated C*-algebra is commutative \cite[Theorem 11.1]{Arv11}, which was later generalized in \cite[Theorem 1.4]{Clo18a}. Despite significant progress made towards establishing the conjecture \cite{DS18, Clo18b, Kim21, CS+, Tho+}, it remained open until now even for type I C*-algebras and continues to be a source of perplexity when the C*-algebra is commutative.

Let $G$ be a separable subset of bounded operators on Hilbert space generating a norm-closed self-adjoint unital subspace $\mathcal{S}$, which in turn generates a C*-algebra $\mathcal{B}$. Arveson's hyperrigidity conjecture \cite[Conjecture 4.3]{Arv11} states that if the restriction to $\mathcal{S}$ of all irreducible $*$-representations of $\mathcal{B}$ have unique completely positive extensions to $\mathcal{B}$, then $G$ is \emph{hyperrigid} in the sense that for every faithful $*$-representation on Hilbert space $\mathcal{B}\subseteq \mathbb{B}(\mathcal{H})$, and a sequence $\{\phi_n\}_{n \in \mathbb{N}}$ of unital completely positive maps on $\mathbb{B}(\mathcal{H})$ we have,
$$
\phi_n(g) \rightarrow g, \ \forall g\in G \ \implies \ \phi_n(b) \rightarrow b, \ \forall b \in \mathcal{B}.
$$

In this paper we disprove Arveson's hyperrigidity conjecture. Although the conjecture was made in terms of operator systems, we may restrict our attention to unital operator algebras (see Section \ref{s:background}).

\vspace{6pt}

\begin{theorem} \label{t:hr-fails}
There is an operator $T$ on Hilbert space such that the the unital operator algebra $\mathcal{A}$ generated by $T$ is not hyperrigid in $\mathcal{B} = C^*(\mathcal{A})$, while the restriction to $\mathcal{A}$ of all irreducible $*$-representations of $\mathcal{B}$ have unique completely positive extensions to $\mathcal{B}$. In fact, $T$ is a rank-one perturbation of a unitary, so that with $G=\mathcal{A}$ we get a counterexample to Arveson's hyperrigidity conjecture with $\mathcal{B}$ of type I.
\end{theorem}

Hyperrigidity was first defined by Arveson in \cite{Arv11} as a non-commuta\-tive analogue of a rigidity of approximations in classical approximation theory \cite{Kor53, Sas67, BK78} (see \cite[Definition 5.2]{DK21}). In \cite[Theorem 2.1]{Arv11} Arveson showed that a unital operator algebra $\mathcal{A}$ is hyperrigid in $\mathcal{B}$ if and only if every $*$-representation $\pi$ of the generated C*-algebra $\mathcal{B}$ has a unique completely positive extension when restricted to $\mathcal{A}$. Irreducible $*$-representations $\pi$ for which $\pi|_{\mathcal{A}}$ has this property are called \emph{boundary representations}, and together they form the \emph{non-commutative Choquet boundary}. The non-commutative Choquet boundary is a generalization of classical Choquet boundary of a uniform algebra, defined as the set of points with unique representing measures (see for instance \cite{Gam69}).

Hence, an obstruction to hyperrigidity occurs if non-commutative Choquet boundary is \emph{properly} contained in the spectrum of $\mathcal{B}$. This obstruction is a non-commuta\-tive analogue of a complete obstruction to rigidity of approximations in the classical setting due to \v{S}a\v{s}kin \cite{Sas67} (see \cite[Theorem 5.3]{DK21} for an extended version), which roughly states that rigidity of approximations holds if and only if the Choquet boundary is the whole space. By analogy, Arveson conjectured \cite[Conjecture 4.3]{Arv11} that the above obstruction to hyperrigidity is the \emph{only} obstruction. 

Theorem \ref{t:hr-fails} illustrates that in order to obtain a non-commutative generalization of \v{S}a\v{s}kin's theorem, hyperrigidity will have to be replaced with a different kind of rigidity of approximations. We discovered our counterexample by attempting to extend results of the second-named author with Salomon \cite{DS18} and describe all boundary representations of tensor algebras in the sense of Muhly and Solel \cite{MS88}. These results will appear in a separate forthcoming paper.

\section{Background} \label{s:background}

We discuss some of the background on non-commutative boundary theory. We refer the reader to \cite{Pau02, Arv+} for some of the necessary prerequisites needed for our paper. A \emph{unital operator algebra} $\mathcal{A}$ is a norm-closed unital subalgebra of bounded operators on Hilbert space, and we denote by $\mathcal{B}=C^*(\mathcal{A})$ the C*-algebra generated by $\mathcal{A}$. A \emph{representation} $\rho : \mathcal{A} \rightarrow \mathbb{B}(\mathcal{H})$ of $\mathcal{A}$ is a unital completely contractive homomorphism. We say a representation $\rho$ has the \emph{unique extension property} if there is a unique completely positive extension of $\rho$ to $\mathcal{B}$, and this extension is multiplicative. An irreducible $*$-representation $\pi$ of $\mathcal{B}$ is called a \emph{boundary representation} if $\pi|_{\mathcal{A}}$ has the unique extension property.

For the purpose of finding a counterexample to Arveson's hyperrigidity conjecture, we may restrict our attention to unital operator algebras instead of operator systems since any representation of a unital operator algebra $\mathcal{A}$ extends \emph{uniquely} to a unital completely positive map on the operator system $\mathcal{S} = \overline{\mathcal{A} + \mathcal{A}^*}$. Hence, a counterexample of a separable unital operator algebra would provides us with a counterexample of a separable operator system generating the same C*-algebra. Moreover, we also have by \cite[Theorem 2.1]{Arv11} that a separable unital operator algebra $\mathcal{A}$ is hyperrigid in its generated C*-algebra $\mathcal{B}$ if and only if for every $*$-representation $\pi$ of $\mathcal{B}$, the restriction $\pi|_{\mathcal{A}}$ has the unique extension property.

Suppose now that $\phi : \mathcal{A} \rightarrow \mathbb{B}(\mathcal{H})$ is a representation of an operator algebra $\mathcal{A}$. A representation $\psi : \mathcal{A} \rightarrow \mathbb{B}(\mathcal{K})$ is said to \emph{dilate} $\phi$ if there is an isometry $V : \mathcal{H} \rightarrow \mathcal{K}$ such that $\phi(a) = V^*\psi(a)V$ for all $a\in \mathcal{A}$. In this way we identify $\mathcal{H}$ as a Hilbert subspace of $\mathcal{K}$, and by Sarason's lemma \cite[Exercise 7.6]{Pau02} we can identify $\mathcal{H}$ as a semi-invariant subspace of $\mathcal{K}$, allowing us to write
$$
\psi(a) = \begin{bmatrix}
* & * & * \\
0 & \phi(a) & * \\
0 & 0 & *
\end{bmatrix}.
$$

The connection between the unique extension property and dilation of representations manifests itself through the notion of maximality. We say that a representation $\phi$ of $\mathcal{A}$ is \emph{maximal} if for any dilation $\psi$ of $\phi$ we have that $\psi = \phi \oplus \phi'$ for some representation $\phi'$. Building on ideas of Agler \cite{Agl88} and of Muhly and Solel \cite{MS98}, Dritschel and McCullough \cite{DM05} showed that a representation $\rho$ of $\mathcal{A}$ has the unique extension property if and only if it is maximal, and that any representation dilates to a maximal representation (see also \cite{Arv+}). These results are of fundamental importance in non-commutative boundary theory

\section{A Type I counterexample} \label{s:typeI}

Our counterexample will be comprised of a unital non-self-adjoint algebra $\mathcal{A}$ generated by a rank-one perturbation of a unitary, and a generated C*-algebra $\mathcal{B}$ of type I. For two vectors $\xi,\eta$ in a Hilbert space, we will denote by $\Theta_{\xi,\eta}$ the rank-one operator given by $\Theta_{\xi,\eta}(\zeta) = \langle \zeta,\eta\rangle \cdot \xi$. Let $U : \ell^2(\mathbb{Z}) \rightarrow \ell^2(\mathbb{Z})$ be the bilateral shift, where $\{e_n\}_{n\in \mathbb{Z}}$ is the standard orthonormal basis. Denote $\mathcal{H} := \ell^2(\mathbb{Z}) \oplus \mathbb{C}\eta$, where $\eta$ is a unit vector, and let $F:=\Theta_{e_0, \eta}$. Define an operator $T$ on $\mathcal{H}$ by
$$
T:= \begin{bmatrix}
U & F \\
0 & 0
\end{bmatrix} = \begin{bmatrix} 
U & 0 \\
0 & 1
\end{bmatrix} + \begin{bmatrix} 
0 & F \\
0 & -1
\end{bmatrix},
$$ 
which is a rank-one perturbation of a unitary on $\mathcal{H}$. 

Let $\mathcal{A}$ be the unital operator algebra generated by $T$ and let $\mathcal{B}$ be the C*-algebra generated by $T$. Note that 
$$
TT^* - 1 = \begin{bmatrix} FF^* & 0 \\ 0 & -1 \end{bmatrix}
$$ 
where $FF^*$ is a projection, so that $\|T\| = \sqrt{2}$, and the projection $P = \frac{(TT^* - 1)^2 - TT^* + 1}{2}$ from $\mathcal{H}$ onto $\mathbb{C}\eta$ is an element in $\mathcal{B}$.

\begin{lemma} \label{l:weakly-to-zero}
For any set $I$ and two vectors $(\xi_i),(\eta_i) \in \oplus_{i\in I}\mathcal{H}$ we have,
$$
\sum_{i\in I}\langle T^n\xi_i,\eta_i \rangle = \langle (\oplus_{i\in I}T^n) (\xi_i),(\eta_i) \rangle \underset{n\rightarrow \infty}{\longrightarrow} 0.
$$
\end{lemma}

\begin{proof}
To prove the lemma, it is equivalent to show that $\{T^n\}_{n\in \mathbb{N}}$ converges to $0$ in the $\sigma$-weak topology. We abuse notation and consider $U$ as an operator on $\mathcal{H}$ by extending it to be $0$ on $\mathbb{C}\eta$. Then, we may write $T^{n+1} = U^nT$, so that $\|T^n\| = \sqrt{2}$ for all $n \in \mathbb{N}$ and the sequence $\{T^n\}_{n\in \mathbb{N}}$ is bounded. Since the $\sigma$-weak and weak operator topologies coincide on bounded sets, it will suffice to show that $\{T^n\}_{n\in \mathbb{N}}$ converges to $0$ in the weak operator topology. Let $\xi,\eta \in \mathcal{H}$, and write
$$
\langle T^{n+1} \xi,\eta \rangle = \langle U^nT\xi,\eta \rangle = \langle U^nT\xi,(1-P)\eta \rangle.
$$
Since $\{U^n\}_{n\in \mathbb{N}}$ on $\ell^2(\mathbb{Z})$ converges to $0$ in the weak operator topology, by taking $\xi' :=T\xi$ and $\eta':=(1-P)\eta$ we get that $\langle U^n\xi',\eta' \rangle \rightarrow 0$. Hence, $\{T^n\}_{n \in \mathbb{N}}$ converges to $0$ in the $\sigma$-weak topology.
\end{proof}

\begin{lemma}
The compact operators $\mathbb{K}(\mathcal{H})$ are an ideal in $\mathcal{B}$. In particular, the identity representation of $\mathcal{B}$ is irreducible.
\end{lemma}

\begin{proof}
First we show that $\eta$ is a cyclic vector for $\mathcal{B}$. Indeed, $\mathcal{B}\eta$ contains the linear span of $\{e_n\}_{n\in \mathbb{N}}$, and since $e_0 \in \mathcal{B}\eta$ we also get that $e_{-1} = (1-P)T^*(e_0) \in \mathcal{B}\eta$, so that $e_{-n} = (T^*)^{-n+1}(e_{-1}) \in \mathcal{B}\eta$ for all $n\in \mathbb{N}$. Hence, $\overline{\mathcal{B}\eta} = \mathcal{H}$, and $\eta$ is cyclic.

Next, since for all $n\in \mathbb{Z}$ we have $e_n \in \mathcal{B}\eta$, there is some $S_n\in \mathcal{B}$ such that $S_n\eta = e_n$, and we get that the rank-one operator $\Theta_{e_n,\eta} = S_nP$ is in $\mathcal{B}$ for any $n \in \mathbb{Z}$. Thus, on top of $\Theta_{e_n,\eta}\in \mathcal{B}$, we also get that $\Theta_{e_n,e_m} = \Theta_{e_n,\eta}\Theta_{e_m,\eta}^* \in \mathcal{B}$, $\Theta_{\eta,\eta} = \Theta_{e_n,\eta}^*\Theta_{e_n,\eta} \in \mathcal{B}$ and $\Theta_{\eta,e_n} =\Theta_{e_n,\eta}^* \in \mathcal{B}$ for all $n,m\in \mathbb{Z}$. Thus, since the linear span of these rank-one operators is dense in $\mathbb{K}(\mathcal{H})$, we get that $\mathbb{K}(\mathcal{H})$ is an ideal of $\mathcal{B}$.
\end{proof}

Since $T$ is a rank-one perturbation of a unitary, by the previous lemma we have the following split short exact sequence
$$
0 \rightarrow \mathbb{K}(\mathcal{H}) \rightarrow \mathcal{B} \rightarrow C(\mathbb{T}) \rightarrow 0.
$$
The C*-algebra $\mathcal{B}$ is an extension of $C(\mathbb{T})$ by an ideal of compact operators $\mathbb{K}(\mathcal{H})$ on separable Hilbert space. Hence, it follows that $\mathcal{B}$ is a type I C*-algebra (see \cite[Theorem 1]{Gli61}).

Let $q : \mathcal{B} \rightarrow C(\mathbb{T})$ be the natural quotient map sending $T$ to the identity function $z\mapsto z$ on $\mathbb{T}$. According to \cite[Theorem 1.3.4]{Arv76} and the paragraph preceding it, the spectrum of $\mathcal{B}$ is made up of the identity representation $\iota : \mathcal{B} \rightarrow \mathbb{B}(\mathcal{H})$ (see \cite[Page 20, Corollary 2]{Arv76}), and evaluations $\tau_z:= \ev_z \circ q : \mathcal{B} \rightarrow \mathbb{B}(\mathbb{C})$ by unimodular scalars $z\in \mathbb{T}$.

\begin{proof}[Proof of Theorem \ref{t:hr-fails}]
$\mathcal{A}$ is generated by $T$ and we saw already that $\mathcal{B} = C^*(T)$ is type I. Thus, we need to show that $\mathcal{A}$ is not hyperrigid in $\mathcal{B}$ while the restriction of $\mathcal{A}$ of all irreducible $*$-representations of $\mathcal{B}$ have the unique extension property.

First, we show that $\mathcal{A}$ is not hyperrigid in $\mathcal{B}$. Due to the equivalence between maximality and the unique extesnion property, it suffices to exhibit a $*$-representation $\sigma$ of $\mathcal{B}$ such that $\sigma|_{\mathcal{A}}$ admits a non-trivial dilation. Let $\tau : C(\mathbb{T}) \rightarrow \mathbb{B}(\ell^2(\mathbb{Z}))$ be the $*$-representation of the bilateral shift given by $z\mapsto U$ which lifts to a $*$-representation $\sigma:= \tau \circ q$ of $\mathcal{B}$. Then, by construction we have for all $A\in \mathcal{A}$ that
$$
\iota(A) = \begin{bmatrix}
\sigma(A) & * \\
0 & * 
\end{bmatrix},
$$
so that $\iota|_{\mathcal{A}}$ dilates $\sigma|_{\mathcal{A}}$. This dilation is clearly non-trivial because $T$ has a non-trivial $(1,2)$ corner, so that $\sigma|_{\mathcal{A}}$ admits a non-trivial dilation.

Second, we show that all irreducible $*$-representations of $\mathcal{B}$ are boundary representations for $\mathcal{A}$. It is clear that the quotient map $q: \mathcal{B} \rightarrow C(\mathbb{T})$ is not isometric on $\mathcal{A}$ because the norm of $T$ becomes strictly smaller. Hence, by Arveson's boundary theorem (see for instance \cite{Dav81}) we have that the identity representation $\iota$ is a boundary representation.

Now, let $z \in \mathbb{T}$ be a unimodular scalar. We will show that $\tau_z : \mathcal{B} \rightarrow \mathbb{B}(\mathbb{C})$ is a boundary representation by showing that $\tau_z|_{\mathcal{A}}$ is maximal. By the Dritschel--McCullough dilation theorem \cite{DM05} (see also \cite{Arv+}) we know that $\tau_z|_{\mathcal{A}}$ admits a dilation to a maximal representation $\rho : \mathcal{A} \rightarrow \mathbb{B}(\mathcal{L})$ with $\xi \in \mathcal{L}$ a unit vector which is the image of $1$ under the dilating isometry $V : \mathbb{C} \rightarrow \mathcal{L}$. Let $\pi$ be the extension of $\rho$ to a $*$-representation, guaranteed by the unique extension property for $\rho$. We assume without loss of generality that $\mathcal{L}$ is the smallest reducing subspace containing $\rho(\mathcal{A})\xi$. Then, by the paragraph preceding \cite[Theorem 1.3.4]{Arv76} we have that $\pi = \pi_{\mathbb{K}} \oplus \pi_{C(\mathbb{T})}$ where $\pi_{\mathbb{K}}$ is a multiple of the identity representation of $\mathcal{B}$ (see \cite[Page 20, Corollary 1]{Arv76}), and $\pi_{C(\mathbb{T})}$ is a representation that annihilates the ideal of compact operators $\mathbb{K}(\mathcal{H})$. Let $\mathcal{L}_{\mathbb{K}} \oplus \mathcal{L}_{\mathbb{T}}$ be the orthogonal decomposition of $\mathcal{L}$ with respect to the decomposition $\pi = \pi_{\mathbb{K}} \oplus \pi_{C(\mathbb{T})}$, and let $Q$ and $R$ be the orthogonal projections onto $\mathcal{L}_{\mathbb{K}}$ and $\mathcal{L}_{\mathbb{T}}$ respectively. 

On the one hand, since $\rho$ dilates $\tau_z$ on $\mathcal{A}$ we have that $|\langle \rho(T^n)\xi,\xi \rangle | = |\tau_z(T^n)| = 1$ for all $n\in \mathbb{N}$. On the other hand if $Q\xi \neq 0$, since the sequence $\{\rho(T^n)\}_{n\in \mathbb{N}}$ on $\mathcal{L}_{\mathbb{K}} = Q \mathcal{L}$ is unitarily equivalent to a multiple of $\{T^n\}_{n \in \mathbb{N}}$ on $\mathcal{H}$, by Lemma \ref{l:weakly-to-zero} we get $\langle \rho(T^n)Q\xi,Q\xi \rangle \rightarrow 0$. Hence,
$$
1 = |\langle \rho(T^n)\xi,\xi \rangle| \leq |\langle \rho(T^n)Q\xi,Q\xi \rangle| + |\langle \rho(T^n)R\xi,R\xi \rangle| \rightarrow 
$$
$$
\limsup_n |\langle \rho(T^n)R\xi,R\xi \rangle| \leq \|R \xi \|^2 < 1,
$$
where the last strict inequality follows by the Pythagorean theorem, and the second to last inequality follows from Cauchy-Schwarz inequality together with the fact that $R\rho(T^n)R = \pi_{C(\mathbb{T})}(T^n)$ is a contraction. Thus, since $\mathcal{L}$ is the smallest reducing subspace containing $\rho(\mathcal{A})\xi$, we must have that $Q \xi = 0$ and $\mathcal{L} = \mathcal{L}_{\mathbb{T}}$, so that $\pi = \pi_{C(\mathbb{T})}$ annihilates $\mathbb{K}(\mathcal{H})$. But now, since $\pi(T) = \rho(T)$ is a \emph{unitary} dilation of the unitary $\tau_z(T)$, we must have that $\tau_z|_{\mathcal{A}}$ is a direct summand of $\rho$. Thus, $ \tau_z|_{\mathcal{A}} = \rho$ is maximal, and $\tau_z$ is a boundary representation for all $z \in \mathbb{T}$.
\end{proof}

This shows that Arveson's hyperrigidity conjecture fails for separable unital commutative operator algebras generating type I C*-algebras. Hence, Arveson's notion of hyperrigidity will need to be replaced with a different kind of non-commutative rigidity of approximations if we are to obtain a non-commutative generalization of \v Sa\v skin's theorem.

\vspace{6pt}

\textbf{Acknowledgments.}
The authors are grateful to Kevin Aguyar Brix, Rapha\"el Clou\^atre, Lucas Hall, Michael Hartz, Ken Davidson, Travis Russell, Guy Salomon and Orr Shalit for comments, remarks and suggestions on preliminary versions of this paper. We especially thank Lucas Hall for organizing the Dilation Theory learning seminar which took place at the Technion in the Fall semester of 2023-2024.


\end{document}